\documentclass{amsart}[12pt]
\usepackage[utf8]{inputenc}
\usepackage{amsmath}
\usepackage{amsfonts}
\usepackage{amssymb}

\newtheorem{theorem}{Theorem}
\newtheorem{lemma}{Lemma}

\newtheorem{corollary}{Corollary}

\theoremstyle{definition}

\begin{document}
\title[Volterra type operators]%
{Volterra type operators on weighted Dirichlet spaces}

\author{Qingze Lin}

\address{School of Mathematics, Sun Yat-sen University, Guangzhou, Guangdong, 510275, P.~R.~China}\email{linqz@mail2.sysu.edu.cn}

\begin{abstract}
The Carleson measures for weighted Dirichlet spaces had been characterized by Girela and Pel\'{a}ez, who also characterized the boundedness of Volterra type operators between weighted Dirichlet spaces. However, their characterizations for the boundedness are not complete. In this paper, we completely characterize the boundedness and compactness of Volterra type operators from the weighted Dirichlet spaces $D_{\alpha}^p$ to $D_{\beta}^q$ ($-1<\alpha,\beta$ and $0<p<q<\infty$), which essentially complete their works. Furthermore, we investigate the order boundedness of Volterra type operators between weighted Dirichlet spaces.
\end{abstract}


\maketitle

\section{\bf Introduction}
Let $\mathbb{D}$ be the unit disk of a complex plane and let $H(\mathbb D)$ be the space consisting of all the analytic functions on $\mathbb{D}$. For $0<p<\infty,-1<\alpha,$ the weighted Bergman space $A_{\alpha}^p$ on the unit disk $\mathbb{D}$ is the space consisting of all the functions $f\in H(\mathbb D)$ such that
$$\|f\|_{A_{\alpha}^p}=\left(\int_{\mathbb{D}}|f(z)|^{p}(1-|z|^2)^\alpha dA(z)\right)^{1/p}<\infty\,,$$
where $dA(z)=\frac{1}{\pi}dxdy$ is the normalized Lebesgue area measure (see \cite{DS,HKZ,zhu} for references). Furthermore, the weighted Dirichlet space $D_{\alpha}^p$ on $\mathbb{D}$ is the space consisting of all the functions $f\in H(\mathbb D)$ satisfying
$$\|f\|_{D_{\alpha}^p}=\left(|f(0)|^p+\int_{\mathbb{D}}|f'(z)|^{p}(1-|z|^2)^\alpha dA(z)\right)^{1/p}<\infty\,.$$

For any fixed function $g\in H(\mathbb D)$, the Volterra type operator $T_g$ and its companion operator $S_g$ are defined, respectively, by
$$(T_gf)(z)=\int_0^z f(\omega)g'(\omega)d\omega\qquad \text{and}\qquad (S_gf)(z)=\int_0^z f'(\omega)g(\omega)d\omega$$
for any $f\in H(\mathbb D)\,.$

Let $|I|$ be the normalized Lebesgue length of $I$, which is an interval of $\partial \mathbb{D}$. The {\it Carleson square} $S(I)$ is defined by
$$S(I):=\{re^{i\theta}:\ e^{i\theta}\in I\,, 1-|I|\leq r<1\}\,.$$

For any $s>0$ and any positive Borel measure $\mu$ in $\mathbb{D}$, we say that $\mu$ is an $s$-Carleson measure if there is a positive constant $C$ such that
$$\mu(S(I))\leq C|I|^s\qquad \text{for all interval }I\subset \partial\mathbb{D}\,.$$

For a space $X$ of analytic functions on $\mathbb{D}$, it is often useful to know the integrability properties of the functions $f\in X$. That is to determine for which positive Borel measure $\mu$ on $\mathbb{D}$ there is a continuous inclusion $X\subset L^p(d\mu)$, or equivalently, by the closed graph theorem, there exists a positive constant $C$ such that for any $f\in X$\,,
$$\|f\|_{L^q(d\mu)}\leq C\|f\|_{X}\,.$$

Duren \cite{Duren} proved that the Hardy space $H^p\subset L^q(d\mu)\,, 0<p\leq q<\infty$, if and only if $\mu$ is a $q/p$-Carleson measure, which extends the result obtained by Carleson \cite{Car} where the case $p=q$ was proven.
For the weighted Bergman spaces, Luecking \cite{Luecking} proved that, for $0<p\leq q<\infty$ and $-1<\alpha$, $A_{\alpha}^p\subset L^q(d\mu)$ if and only if $\mu$ is a $\frac{q(\alpha+2)}{p}$-Carleson measure.

For $0<p<q<\infty$ and $-1<\alpha$, Girela and Pel\'{a}ez \cite{GP} gave the characterizations of the measures $\mu$ for which $D_{\alpha}^p\subset L^q(d\mu)$\,. Indeed, they proved the following theorem:

\begin{theorem}\label{th1}
Suppose that $0<p<q<\infty, -1<\alpha$ and $\mu$ is a positive Borel measure in $\mathbb{D}$, then\\
\phantom{(1)}(1) If $p<\alpha+2$, then $D_{\alpha}^p\subset L^q(d\mu)$ if and only if $\mu$ is a $\frac{q(\alpha+2-p)}{p}$-Carleson measure;\\
\phantom{(1)}(2) If $p=\alpha+2$, then $D_{\alpha}^p\subset L^q(d\mu)$ if and only if there exists a positive constant $C$ such that for all interval $I\subset\partial\mathbb{D}$, it holds that $\mu(S(I))\leq C\left(\log{\frac{1}{|I|}}\right)^{(1/p-1)q}$;\\
\phantom{(1)}(3) If $p>\alpha+2$, then $D_{\alpha}^p\subset L^q(d\mu)$ if and only if $\mu$ is a finite measure.
\end{theorem}

For the case of $p\geq q$, the corresponding characterizations were partly investigated in \cite{GGP,PZ,wu}, where several questions were still open.

In section~2, we completely characterize the boundedness of Volterra type operators $T_g$ and $S_g$ from the weighted Dirichlet spaces $D_{\alpha}^p$ to $D_{\beta}^q$ ($-1<\alpha,\beta$ and $0<p<q<\infty$), which extend the works by Girela and Pel\'{a}ez in \cite{GP}, where the original characterizations only covered the case $\alpha<p<\alpha+2$\,. In section~3, we investigate the compactness of the Volterra type operators $T_g$ and $S_g$ from $D_{\alpha}^p$ to $D_{\beta}^q$ ($-1<\alpha,\beta$ and $0<p<q<\infty$). Finally, in section~4, we investigate the order boundedness of Volterra type operators between weighted Dirichlet spaces. Throughout the paper, $C$ will represent a positive constant which may be different at different occurrences.

\section{\bf Boundedness of Volterra type operators}
The Volterra type operator $T_g$ was introduced by Pommerenke \cite{P} to study the exponentials of BMOA functions and in the meantime, he proved that $T_g$ acting on the Hardy-Hilbert space $H^2$ is bounded if and only if $g\in BMOA$\,. After his work, Aleman, Siskakis and Cima \cite{AC,AS} studied the boundedness and compactness of $T_g$ on the Hardy space $H^p$, where they showed that $T_g$ is bounded (compact) on $H^p,\ 0<p<\infty$, if and only if $g\in BMOA\ (g\in VMOA)$. For the related works, see \cite{LMN}. Furthermore, Aleman and Siskakis \cite{AS1} studied the boundedness and compactness of $T_g$ on the Bergman spaces while Galanopoulos et al. \cite{GGP,GP} investigated the boundedness of $T_g$ and $S_g$ on the Dirichlet type spaces, and Xiao \cite{XJ} studied the Volterra type operators on $Q_p$ spaces through the characterizations of the Carleson measures. It should be noted that Li, Liu and Lou \cite{LLL} dealt with $T_g$ and $S_g$ operators whose range is the Morrey space and whose domain is either the Hardy space or the Morrey space.

Recently, Lin et al. \cite{LIN1,LIN,LIN2} characterized the boundedness and the strict singularities of the Volterra type operators acting on the (derivative) Hardy spaces and weighted Banach spaces with general weights. Li and Stevi\'{c} \cite{Li1,Li2} introduced the generalized composition operators (also called generalized Volterra type operators) acting on Zygmund spaces and Bloch type spaces and so forth, which had attracted intensive attentions. For instance, Mengestie \cite{TM} obtained a complete description of the boundedness and compactness of the product of the Volterra type operators and composition operators on the weighted Fock spaces, and recently, he \cite{TM1}  studied the topological structure of the space of Volterra-type integral operators on the Fock spaces endowed with the operator norm. Furthermore, by applying the Carleson embedding theorem and the Littlewood-Paley formula, Constantin and Pel\'{a}ez \cite{CP} obtained the boundedness and compactness of $T_g$ on the weighted Fock spaces and investigated the invariant subspaces of the classical Volterra operator $T_z$ on such spaces.

The multiplication operator $M_g$ is defined by
$$(M_gf)(z):=g(z)f(z)\,,\quad \text{for }f\in H(\mathbb{D}), z\in \mathbb{D}\,.$$
The following relation holds:
$$(M_gf)(z)=f(0)g(0)+(T_gf)(z)+(S_gf)(z)\,.$$

Then we characterize the boundedness of these operators.

\begin{theorem}\label{th2}
Let $-1<\alpha,\beta$, $g\in H(\mathbb D)$ and $0<p<q<\infty$\,. Define $d\mu_{g,q,\beta}(z):=(1-|z|^2)^\beta|g'(z)|^qdA(z)$. Then the following statements hold:\\
\phantom{(1)}(1) If $p<\alpha+2$, then $T_g:D_{\alpha}^p\rightarrow D_{\beta}^q$ is bounded if and only if $\mu_{g,q,\beta}(z)$ is a $\frac{q(\alpha+2-p)}{p}$-Carleson measure;\\
\phantom{(1)}(2) If $p=\alpha+2$, then $T_g:D_{\alpha}^p\rightarrow D_{\beta}^q$ is bounded if and only if there exists a positive constant $C$ such that for all interval $I\subset\partial\mathbb{D}$, it holds that $\mu_{g,q,\beta}(S(I))\leq C\left(\log{\frac{1}{|I|}}\right)^{(1/p-1)q}$;\\
\phantom{(1)}(3) If $p>\alpha+2$, then $T_g:D_{\alpha}^p\rightarrow D_{\beta}^q$ is bounded if and only if $\mu_{g,q,\beta}$ is a finite measure, or equivalently, $g\in D_{\beta}^q$\,.
\end{theorem}
\begin{proof}
This follows directly from Theorem~\ref{th1} and the closed graph theorem.
\end{proof}

\begin{theorem}\label{th3}
Let $-1<\alpha,\beta$, $g\in H(\mathbb D)$ and $0<p<q<\infty$\,. Then $S_g:D_{\alpha}^p\rightarrow D_{\beta}^q$ is bounded if and only if $|g(z)|=O\left((1-|z|^2)^{\frac{2+\alpha}{p}-\frac{2+\beta}{q}}\right)\,,$ as $|z|\rightarrow1^-$\,.
\end{theorem}
\begin{proof}
First, suppose that $|g(z)|=O\left((1-|z|^2)^{\frac{2+\alpha}{p}-\frac{2+\beta}{q}}\right)\,.$
If $f\in D_{\alpha}^p$, then $f'\in A_{\alpha}^p$ by definition. It is a well-known fact (see \cite{DS,zhu}) that if $h\in A_{\alpha}^p$, then for all $z\in\mathbb{D}$, we have
$$|h(z)|\leq C\frac{\|h\|_{A_{\alpha}^p}}{(1-|z|^2)^{(\alpha+2)/p}}\,.$$
Then it holds that
\begin{equation}\begin{split}\nonumber
\|S_gf\|_{D_{\beta}^q}&=\left(\int_{\mathbb{D}}|f'(z)g(z)|^q(1-|z|^2)^\beta dA(z)\right)^{1/q}\\
&\leq C \left(\int_{\mathbb{D}}|f'(z)|^p|f'(z)|^{q-p}(1-|z|^2)^{\frac{q(2+\alpha)}{p}-2} dA(z)\right)^{1/q}\\
&\leq C \left(\int_{\mathbb{D}}|f'(z)|^p\left(\frac{\|f\|_{D_{\alpha}^p}}{(1-|z|^2)^{\frac{(2+\alpha)}{p}}}\right)^{q-p}(1-|z|^2)^{\frac{q(2+\alpha)}{p}-2} dA(z)\right)^{1/q}\\
&\leq C\|f\|^{(q-p)/q}_{D_{\alpha}^p} \left(\int_{\mathbb{D}}|f'(z)|^p(1-|z|^2)^{\alpha} dA(z)\right)^{1/q}\\
&\leq C\|f\|_{D_{\alpha}^p}\,.
\end{split}\end{equation}
Hence, $S_g:D_{\alpha}^p\rightarrow D_{\beta}^q$ is bounded.

Conversely, suppose that $S_g:D_{\alpha}^p\rightarrow D_{\beta}^q$ is bounded. Given $a\in\mathbb{D}$, define the function $f_a$ by
$$f_a(z):=\frac{(1-|a|^2)^{(\alpha+2)/p}}{(1-\bar{a}z)^{2(\alpha+2)/p-1}}\,.$$
It is easy to prove that $f_a\in D_{\alpha}^p$ and there exists a positive constant $C$ such that for all $a\in\mathbb{D}$, $\|f_a\|_{D_{\alpha}^p}\leq C\,.$
Denoting $\Delta(a,r)$ as the pseudo-hyperbolic disk with center $a$ and radius $r$, we have
\begin{equation}\begin{split}\nonumber
(1-|a|^2)^{{2+\beta}-\frac{q(2+\alpha)}{p}}|g(a)|^q&\leq C(1-|a|^2)^{{\beta}-\frac{q(2+\alpha)}{p}}\int_{\Delta(a,r)}|g(\omega)|^qdA(\omega)\\
&\leq C|a|^{-q}\int_{\Delta(a,r)}|(S_gf_a)'(\omega)|^q(1-|\omega|^2)^{\beta}dA(\omega)\\
&\leq C|a|^{-q}\|S_gf_a\|^q_{D_{\beta}^q}\\
&\leq C|a|^{-q}\|S_g\|^q\|f_a\|^q_{D_{\alpha}^p}\\
&\leq C|a|^{-q}\,.
\end{split}\end{equation}
Thus, $|g(a)|=O\left((1-|a|^2)^{\frac{2+\alpha}{p}-\frac{2+\beta}{q}}\right)\,,$ as $|a|\rightarrow1^-$\,.
\end{proof}

As an immediate corollary, we obtain the known results originally proven by Zhao \cite{zhao}\,.
\begin{corollary}
Let $-1<\alpha,\beta$, $g\in H(\mathbb D)$ and $0<p<q<\infty$\,. Then $M_g:A_{\alpha}^p\rightarrow A_{\beta}^q$ is bounded if and only if $|g(z)|=O\left((1-|z|^2)^{\frac{2+\alpha}{p}-\frac{2+\beta}{q}}\right)\,,$ as $|z|\rightarrow1^-$\,.
\end{corollary}
\begin{proof}
This follows immediately from the fact that $DS_g=M_gD$, where $D$ is the differential operator.
\end{proof}

\begin{theorem}\label{th4}
Let $-1<\alpha,\beta$, $g\in H(\mathbb D)$ and $0<p<q<\infty$\,. Define $d\mu_{g,q,\beta}(z):=(1-|z|^2)^\beta|g'(z)|^qdA(z)$. Then the following statements hold:\\
\phantom{(1)}(1) If $p<\alpha+2$, then $M_g:D_{\alpha}^p\rightarrow D_{\beta}^q$ is bounded if and only if $\mu_{g,q,\beta}(z)$ is a $\frac{q(\alpha+2-p)}{p}$-Carleson measure and $|g(z)|=O\left((1-|z|^2)^{\frac{2+\alpha}{p}-\frac{2+\beta}{q}}\right)\,,$ as $|z|\rightarrow1^-$;\\
\phantom{(1)}(2) If $p=\alpha+2$, then $M_g:D_{\alpha}^p\rightarrow D_{\beta}^q$ is bounded if and only if $|g(z)|=O\left((1-|z|^2)^{\frac{2+\alpha}{p}-\frac{2+\beta}{q}}\right)$ as $|z|\rightarrow1^-$ and there exists a positive constant $C$ such that for all interval $I\subset\partial\mathbb{D}$, it holds that $\mu_{g,q,\beta}(S(I))\leq C\left(\log{\frac{1}{|I|}}\right)^{(1/p-1)q}$;\\
\phantom{(1)}(3) If $p>\alpha+2$, then $M_g:D_{\alpha}^p\rightarrow D_{\beta}^q$ is bounded if and only if $|g(z)|=O\left((1-|z|^2)^{\frac{2+\alpha}{p}-\frac{2+\beta}{q}}\right)$ as $|z|\rightarrow1^-$ and $g\in D_{\beta}^q$\,.
\end{theorem}
\begin{proof}
Since $(M_gf)(z)=f(0)g(0)+(T_gf)(z)+(S_gf)(z)$\,, the sufficiency follows immediately from Theorem~\ref{th2} and Theorem~\ref{th3}. It remains to prove the necessity. In this case, it is obvious that if we can prove that $|g(z)|=O\left((1-|z|^2)^{\frac{2+\alpha}{p}-\frac{2+\beta}{q}}\right)$ as $|z|\rightarrow1^-$, then all the other statements follow immediately from Theorem~\ref{th2} and Theorem~\ref{th3} again.

Given $a\in\mathbb{D}$, define the function $F_a$ by
$$F_a(z):=\frac{(1-|a|^2)^{(\alpha+2)/p}}{(1-\bar{a}z)^{2(\alpha+2)/p-1}}-(1-|a|^2)^{(p-\alpha-2)/p}\,.$$
Then $F_a(a)=0$, and the remainder of the proof is essentially similar to the converse part of the proof in Theorem~\ref{th3}.
\end{proof}

\section{\bf Compactness of Volterra type operators}
For any $s>0$ and $\mu$ a positive Borel measure in $\mathbb{D}$, we say $\mu$ is a vanishing $s$-Carleson measure if
$$\mu(S(I))=o(|I|^s)\qquad \text{as }|I|\rightarrow0\,.$$

\begin{theorem}\label{th5}
Suppose that $0<p<q<\infty, -1<\alpha$ and $\mu$ is a positive Borel measure in $\mathbb{D}$, then\\
\phantom{(1)}(1) If $p<\alpha+2$, then $D_{\alpha}^p\subset L^q(d\mu)$ is compact if and only if $\mu$ is a vanishing $\frac{q(\alpha+2-p)}{p}$-Carleson measure;\\
\phantom{(1)}(2) If $p=\alpha+2$, then $D_{\alpha}^p\subset L^q(d\mu)$ is compact if and only if $\mu(S(I))= o\left(\left(\log{\frac{1}{|I|}}\right)^{(1/p-1)q}\right)$ as $|I|\rightarrow0$;\\
\phantom{(1)}(3) If $p>\alpha+2$, then $D_{\alpha}^p\subset L^q(d\mu)$ is compact if and only if $\mu$ is a finite measure.
\end{theorem}
\begin{proof}
(1) is known (see, for example, \cite{Kumar}).

For (2), we noticed that this condition is, in deed, a vanishing $\left((1-1/p)q,0\right)$-logarithmic Carleson measure and the proof of it is basically similar to (ii) of Theorem~3.1 in \cite{PZ}\,.

Now for (3), since when $p>\alpha+2$, it holds that $D_{\alpha}^p\subset H^\infty$, where $H^\infty$ is the space of all the bounded analytic functions on $\mathbb{D}$, then the compactness follows easily by the standard arguments.
\end{proof}

Then we characterize the compactness of these operators.

\begin{theorem}\label{th6}
Let $-1<\alpha,\beta$, $g\in H(\mathbb D)$ and $0<p<q<\infty$\,. Define $d\mu_{g,q,\beta}(z):=(1-|z|^2)^\beta|g'(z)|^qdA(z)$. Then the following statements hold:\\
\phantom{(1)}(1) If $p<\alpha+2$, then $T_g:D_{\alpha}^p\rightarrow D_{\beta}^q$ is compact if and only if $\mu_{g,q,\beta}(z)$ is a vanishing $\frac{q(\alpha+2-p)}{p}$-Carleson measure;\\
\phantom{(1)}(2) If $p=\alpha+2$, then $T_g:D_{\alpha}^p\rightarrow D_{\beta}^q$ is compact if and only if $\mu_{g,q,\beta}(S(I))=o\left(\left(\log{\frac{1}{|I|}}\right)^{(1/p-1)q}\right)$ as $|I|\rightarrow0$;\\
\phantom{(1)}(3) If $p>\alpha+2$, then $T_g:D_{\alpha}^p\rightarrow D_{\beta}^q$ is compact if and only if $\mu_{g,q,\beta}$ is a finite measure, or equivalently, $g\in D_{\beta}^q$\,.
\end{theorem}
\begin{proof}
This follows directly from Theorem~\ref{th5}.
\end{proof}

\begin{theorem}\label{th7}
Let $-1<\alpha,\beta$, $g\in H(\mathbb D)$ and $0<p<q<\infty$\,. Then $S_g:D_{\alpha}^p\rightarrow D_{\beta}^q$ is compact if and only if $|g(z)|=o\left((1-|z|^2)^{\frac{2+\alpha}{p}-\frac{2+\beta}{q}}\right)\,,$ as $|z|\rightarrow1^-$\,.
\end{theorem}
\begin{proof}
First suppose that $|g(z)|=o\left((1-|z|^2)^{\frac{2+\alpha}{p}-\frac{2+\beta}{q}}\right)\,.$
Then, for any $\epsilon>0$, there exists $r$ with $0<r<1$ such that $\frac{|g(z)|}{\left((1-|z|^2)^{\frac{2+\alpha}{p}-\frac{2+\beta}{q}}\right)}<\epsilon\,,$ whenever $|z|>r$\,.
Now, for any bounded sequence $\{f_n\}^\infty_{n=0}\subset D_{\alpha}^p$ such that $f_n$ converges to 0 locally uniformly, it holds that
\begin{equation}\begin{split}\nonumber
&\phantom{=\;}\limsup_{n\rightarrow\infty}\|S_gf_n\|_{D_{\beta}^q}\\
&=\limsup_{n\rightarrow\infty}\left(\int_{\mathbb{D}}|f'_n(z)g(z)|^q(1-|z|^2)^\beta dA(z)\right)^{1/q}\\
&\leq \limsup_{n\rightarrow\infty}\left(\int_{\mathbb{D}\backslash r\overline{\mathbb{D}}}|f'_n(z)g(z)|^q(1-|z|^2)^\beta dA(z)\right)^{1/q}\\
&\leq\limsup_{n\rightarrow\infty} C\epsilon^{1/q} \left(\int_{\mathbb{D}}|f'_n(z)|^p|f'(z)|^{q-p}(1-|z|^2)^{\frac{q(2+\alpha)}{p}-2} dA(z)\right)^{1/q}\\
&\leq\limsup_{n\rightarrow\infty} C\epsilon^{1/q}\left(|f'_n(z)|^p\left(\frac{\|f_n\|_{D_{\alpha}^p}}{(1-|z|^2)^{\frac{(2+\alpha)}{p}}}\right)^{q-p}(1-|z|^2)^{\frac{q(2+\alpha)}{p}-2} dA(z)\right)^{1/q}\\
&\leq\limsup_{n\rightarrow\infty} C\epsilon^{1/q}\|f_n\|^{(q-p)/q}_{D_{\alpha}^p} \left(\int_{\mathbb{D}}|f'_n(z)|^p(1-|z|^2)^{\alpha} dA(z)\right)^{1/q}\\
&\leq\limsup_{n\rightarrow\infty} C\epsilon^{1/q}\|f_n\|_{D_{\alpha}^p}\\
&\leq C\epsilon^{1/q}\,.
\end{split}\end{equation}
Since $\epsilon$ is arbitrary, it follows that $S_g:D_{\alpha}^p\rightarrow D_{\beta}^q$ is compact.

Conversely, suppose that $S_g:D_{\alpha}^p\rightarrow D_{\beta}^q$ is compact. Choose the functions $f_a$ defined in the proof of Theorem~\ref{th3}, then the direct computation shows that $\|f_a\|_{D_{\alpha}^p}$ is uniformly bounded for all $a\in\mathbb{D}$  and $f_a$ converges to 0 locally uniformly in $\mathbb{D}$. Thus, we have
\begin{equation}\begin{split}\nonumber
(1-|a|^2)^{{2+\beta}-\frac{q(2+\alpha)}{p}}|g(a)|^q&\leq C(1-|a|^2)^{{\beta}-\frac{q(2+\alpha)}{p}}\int_{\Delta(a,r)}|g(\omega)|^qdA(\omega)\\
&\leq C|a|^{-q}\int_{\Delta(a,r)}|(S_gf_a)'(\omega)|^q(1-|\omega|^2)^{\beta}dA(\omega)\\
&\leq C|a|^{-q}\|S_gf_a\|^q_{D_{\beta}^q}\rightarrow0 \qquad\text{ as }|a|\rightarrow1^-\,.
\end{split}\end{equation}
Thus, $|g(a)|=o\left((1-|a|^2)^{\frac{2+\alpha}{p}-\frac{2+\beta}{q}}\right)\,,$ as $|a|\rightarrow1^-$\,.
\end{proof}

As an immediate corollary, we obtain the known results originally proven by \v{C}u\v{c}kovi\'{c} and Zhao \cite{ZCRZ}\,.

\begin{corollary}
Let $-1<\alpha,\beta$, $g\in H(\mathbb D)$ and $0<p<q<\infty$\,. Then $M_g:A_{\alpha}^p\rightarrow A_{\beta}^q$ is compact if and only if $|g(z)|=o\left((1-|z|^2)^{\frac{2+\alpha}{p}-\frac{2+\beta}{q}}\right)\,,$ as $|z|\rightarrow1^-$\,.
\end{corollary}

\begin{theorem}\label{th8}
Let $-1<\alpha,\beta$, $g\in H(\mathbb D)$ and $0<p<q<\infty$\,. Define $d\mu_{g,q,\beta}(z):=(1-|z|^2)^\beta|g'(z)|^qdA(z)$. Then the following statements hold:\\
\phantom{(1)}(1) If $p<\alpha+2$, then $M_g:D_{\alpha}^p\rightarrow D_{\beta}^q$ is compact if and only if $\mu_{g,q,\beta}(z)$ is a vanishing $\frac{q(\alpha+2-p)}{p}$-Carleson measure and $|g(z)|=o\left((1-|z|^2)^{\frac{2+\alpha}{p}-\frac{2+\beta}{q}}\right)\,,$ as $|z|\rightarrow1^-$;\\
\phantom{(1)}(2) If $p=\alpha+2$, then $M_g:D_{\alpha}^p\rightarrow D_{\beta}^q$ is compact if and only if $|g(z)|=o\left((1-|z|^2)^{\frac{2+\alpha}{p}-\frac{2+\beta}{q}}\right)$ as $|z|\rightarrow1^-$ and $\mu_{g,q,\beta}(S(I))=o\left(\left(\log{\frac{1}{|I|}}\right)^{(1/p-1)q}\right)$ as $|I|\rightarrow0$;\\
\phantom{(1)}(3) If $p>\alpha+2$, then $M_g:D_{\alpha}^p\rightarrow D_{\beta}^q$ is compact if and only if $|g(z)|=o\left((1-|z|^2)^{\frac{2+\alpha}{p}-\frac{2+\beta}{q}}\right)$ as $|z|\rightarrow1^-$ and $g\in D_{\beta}^q$\,.
\end{theorem}
\begin{proof}
Since $(M_gf)(z)=f(0)g(0)+(T_gf)(z)+(S_gf)(z)$\,, the sufficiency follows immediately from Theorem~\ref{th6} and Theorem~\ref{th7}. It remains to prove the necessary conditions and in this case, it is obvious that if we can prove that $|g(z)|=o\left((1-|z|^2)^{\frac{2+\alpha}{p}-\frac{2+\beta}{q}}\right)$ as $|z|\rightarrow1^-$, then all the other statements follow immediately from Theorem~\ref{th6} and Theorem~\ref{th7} again.

Given $a\in\mathbb{D}$, define the function $F_a$ by
$$F_a(z):=\frac{(1-|a|^2)^{(\alpha+2)/p}}{(1-\bar{a}z)^{2(\alpha+2)/p-1}}-(1-|a|^2)^{(p-\alpha-2)/p}\,.$$
Then $F_a(a)=0$, and the remainder of the proof is similar to that of Theorem~\ref{th7}.
\end{proof}

\section{\bf Order boundedness of Volterra type operators}
Let $X$ be a Banach space of holomorphic functions defined on $\mathbb{D}$\,, $q>0$\,, $(\Omega,\mathcal{A},\mu)$ a measure space and
$$L^p(\Omega,\mathcal{A},\mu):=\{f|\ f:\Omega\to\mathbb{C} \text{ is measurable and } \int_{\Omega}|f|^pd\mu<\infty\}\,.$$
An operator $T:X\to L^p(\Omega,\mathcal{A},\mu)$ is said to be order bounded if there exists a nonnegative function $g\in L^p(\Omega,\mathcal{A},\mu)$ such that for all $f\in X$ with $\|f\|_X\leq 1$, it holds that
$$|T(f)(x)|\leq g(x)\,, \quad \text{ a.e. } [\mu]\,.$$

Order boundedness  plays an important role in studying the properties of many concrete operators acting between Banach spaces like Hardy spaces, weighted Bergman spaces and so forth (see \cite{RH,HJ,SU,WWG}). Recently, order boundedness of weighted composition operators between weighted Dirichlet spaces were studied in \cite{GKZ,AS}\,. In this section, we investigate the order boundedness of Volterra type operators between weighted Dirichlet spaces. Recall that in this case, if we define the measure $A_{\beta}$ by $dA_{\beta}(z)=(1-|z|^2)^\beta dA(z)$, then an operator $T:D_{\alpha}^p\rightarrow D_{\beta}^q$ is order bounded if and only if there exists a nonnegative function $g\in L^q(A_{\beta})$ such that for all $f\in D_{\alpha}^p$ with $\|f\|_{D_{\alpha}^p}\leq 1$, it holds that
$$|T(f)'(z)|\leq g(z)\,, \quad \text{ a.e. } [A_{\beta}]\,.$$

Before proving the results, we first give some auxiliary lemmas.

\begin{lemma}\label{le1}
Let $\alpha>-1$ and $0<p<\infty$. Denote $\delta_z$ as the point evaluation functional on $D_{\alpha}^p$, then

(1) for $p<\alpha+2$,\quad $\|\delta_z\|\approx \frac{1}{(1-|z|^2)^{(\alpha+2-p)/p}}$;

(2) for $p=\alpha+2$,\quad $\|\delta_z\|\approx \frac{1}{\left(\log(\frac{2}{1-|z|^2})\right)^{(1-p)/p}}$;

(3) for $p>\alpha+2$,\quad $\|\delta_z\|\approx 1$\,.
\end{lemma}
\begin{proof}
(1) and (2) follows from \cite[Lemma~2.2 and Lemma~2.3]{GKZ} while (3) follows directly from the fact that $D_{\alpha}^p\subset H^\infty$ for $p>\alpha+2$\,.
\end{proof}

\begin{lemma}\label{le2}
Let $\alpha>-1$ and $0<p<\infty$. Denote $\delta'_z$ as the derivative point evaluation functional on $D_{\alpha}^p$, then
$\|\delta'_z\|\approx \frac{1}{(1-|z|^2)^{(\alpha+2)/p}}\,.$
\end{lemma}
\begin{proof}
By definition, $f\in D_{\alpha}^p$ if and only if $f'\in A_{\alpha}^p$, thus the lemma follows from \cite[Lemma~3.2]{HKZ}.
\end{proof}

Now we are ready to prove our results.
\begin{theorem}\label{th9}
Let $-1<\alpha,\beta$, $g\in H(\mathbb D)$ and $0<p,q<\infty$\,. Then the following statements hold:\\
\phantom{(1)}(1) If $p<\alpha+2$, then $T_g:D_{\alpha}^p\rightarrow D_{\beta}^q$ is order bounded if and only if $$\int_{\mathbb{D}}\frac{|g'(z)|^q}{(1-|z|^2)^{q(\alpha+2-p)/p}}dA_{\beta}<\infty\,;$$
\phantom{(1)}(2) If $p=\alpha+2$, then $T_g:D_{\alpha}^p\rightarrow D_{\beta}^q$ is order bounded if and only if $$\int_{\mathbb{D}}\frac{|g'(z)|^q}{\left(\log(\frac{2}{1-|z|^2})\right)^{q(1-p)/p}}dA_{\beta}<\infty\,;$$
\phantom{(1)}(3) If $p>\alpha+2$, then $T_g:D_{\alpha}^p\rightarrow D_{\beta}^q$ is order bounded if and only if $g\in D_{\beta}^q$\,.
\end{theorem}
\begin{proof}
(1) Assume first that $T_g:D_{\alpha}^p\rightarrow D_{\beta}^q$ is order bounded. Then there exists $h\in L^q(A_{\beta})$ such that for all $f\in D_{\alpha}^p$ with $\|f\|_{D_{\alpha}^p}\leq 1$, it holds that
$$|f(z)g'(z)|\leq h(z)\,, \quad \text{ a.e. } [A_{\beta}]\,.$$
Hence, by Lemma~\ref{le1}, the inequality
\begin{equation}\begin{split}\nonumber
h(z)\geq |g'(z)|\|\delta_z\|\gtrsim \frac{|g'(z)|}{(1-|z|^2)^{(\alpha+2-p)/p}}\quad \text{ holds a.e. }[A_{\beta}]\,.
\end{split}\end{equation}
Therefore, it holds that $\int_{\mathbb{D}}\frac{|g'(z)|^q}{(1-|z|^2)^{q(\alpha+2-p)/p}}dA_{\beta}<\infty\,.$

Conversely, suppose that $\int_{\mathbb{D}}\frac{|g'(z)|^q}{(1-|z|^2)^{q(\alpha+2-p)/p}}dA_{\beta}<\infty\,.$ Let
$$h(z)=\frac{|g'(z)|}{(1-|z|^2)^{(\alpha+2-p)/p}}\,,$$
then by Lemma~\ref{le1}, for all $f\in D_{\alpha}^p$ with $\|f\|_{D_{\alpha}^p}\leq 1$, $$|f(z)g'(z)|\leq |g'(z)|\|\delta_z\|\lesssim h(z)\,, \quad \text{ a.e. } [A_{\beta}]\,.$$
Therefore, $T_g:D_{\alpha}^p\rightarrow D_{\beta}^q$ is order bounded.

The proof of (2) and (3) is almost similar to that of (1), thus we omit the details.
\end{proof}

By Theorem~\ref{th2}, Theorem~\ref{th6} and Theorem~\ref{th9}, we obtain the following corollary.
\begin{corollary}\label{cor3}
Let $-1<\alpha,\beta$, $g\in H(\mathbb D)$ and $\alpha+2<p<q<\infty$\,. Then the following statements are equivalent:\\
\phantom{(1)}(1) $T_g:D_{\alpha}^p\rightarrow D_{\beta}^q$ is bounded;\\
\phantom{(1)}(2) $T_g:D_{\alpha}^p\rightarrow D_{\beta}^q$ is compact;\\
\phantom{(1)}(3) $T_g:D_{\alpha}^p\rightarrow D_{\beta}^q$ is order bounded;\\
\phantom{(1)}(4) $g\in D_{\beta}^q$\,.
\end{corollary}

\begin{theorem}\label{th10}
Let $-1<\alpha,\beta$, $g\in H(\mathbb D)$ and $0<p,q<\infty$\,. Then
$S_g:D_{\alpha}^p\rightarrow D_{\beta}^q$ is order bounded if and only if $$\int_{\mathbb{D}}\frac{|g(z)|^q}{(1-|z|^2)^{q(\alpha+2)/p}}dA_{\beta}<\infty\,.$$
\end{theorem}
\begin{proof}
The proof is similar to that of Theorem~\ref{th9} except that in this case, we resort to Lemma~\ref{le2} instead of Lemma~\ref{le1}.
\end{proof}

\begin{theorem}\label{th11}
Let $-1<\alpha,\beta$, $g\in H(\mathbb D)$ and $0<p,q<\infty$\,. Then the following statements hold:\\
\phantom{(1)}(1) If $p<\alpha+2$, then $M_g:D_{\alpha}^p\rightarrow D_{\beta}^q$ is order bounded if and only if
$$\int_{\mathbb{D}}\frac{|g(z)|^q}{(1-|z|^2)^{q(\alpha+2)/p}}dA_{\beta}<\infty\,;$$
\phantom{(1)}(2) If $p=\alpha+2$, then $M_g:D_{\alpha}^p\rightarrow D_{\beta}^q$ is order bounded if and only if $$\int_{\mathbb{D}}\frac{|g(z)|^q}{(1-|z|^2)^{q}}dA_{\beta}+\int_{\mathbb{D}}\frac{|g'(z)|^q}{\left(\log(\frac{2}{1-|z|^2})\right)^{q(1-p)/p}}dA_{\beta}<\infty\,;$$
\phantom{(1)}(3) If $p>\alpha+2$, then $M_g:D_{\alpha}^p\rightarrow D_{\beta}^q$ is order bounded if and only if $g\in D_{\beta}^q$ and
$$\int_{\mathbb{D}}\frac{|g(z)|^q}{(1-|z|^2)^{q(\alpha+2)/p}}dA_{\beta}<\infty\,.$$
\end{theorem}
\begin{proof}
(1) Suppose that $\int_{\mathbb{D}}\frac{|g(z)|^q}{(1-|z|^2)^{q(\alpha+2)/p}}dA_{\beta}<\infty\,.$
Let $f\in D_{\alpha}^p$ with $\|f\|_{D_{\alpha}^p}\leq1$, then by Lemma~\ref{le1} and Lemma~\ref{le2}, we have
$$|(f(z)g(z))'|\leq|f'(z)g(z)|+|f(z)g'(z)|\lesssim \frac{|g(z)|}{(1-|z|^2)^{\frac{(\alpha+2)}{p}}}+\frac{|g'(z)|}{(1-|z|^2)^{\frac{(\alpha+2-p)}{p}}}\,.$$
By taking
$$h(z)=\frac{|g(z)|}{(1-|z|^2)^{\frac{(\alpha+2)}{p}}}+\frac{|g'(z)|}{(1-|z|^2)^{\frac{(\alpha+2-p)}{p}}}\,,$$
then $h\in L^q(A_{\beta})$ since
$$\int_{\mathbb{D}}\frac{|g'(z)|}{(1-|z|^2)^{\frac{(\alpha+2-p)}{p}}}dA_{\beta}\lesssim \int_{\mathbb{D}}\frac{|g(z)|}{(1-|z|^2)^{\frac{(\alpha+2)}{p}}}dA_{\beta}<\infty\,.$$
Accordingly, $M_g:D_{\alpha}^p\rightarrow D_{\beta}^q$ is order bounded.

Conversely, assume that $M_g:D_{\alpha}^p\rightarrow D_{\beta}^q$ is order bounded. Then there exists $h\in L^q(A_{\beta})$ such that for all $f\in D_{\alpha}^p$ with $\|f\|_{D_{\alpha}^p}\leq 1$, it holds that
$$|(fg)'(z)|\leq h(z)\,, \quad \text{ a.e. } [A_{\beta}]\,.$$
For any $z\in\mathbb{D}$, we consider the function
$$f_z(\omega)=\frac{(1-|z|^2)^{(\alpha+2)/p}}{(1-\bar{z}\omega)^{2(\alpha+2)/p-1}}-\frac{(1-|z|^2)^{(\alpha+2)/p+1}}{(1-\bar{z}\omega)^{2(\alpha+2)/p}}\,,\quad \omega\in\mathbb{D}\,.$$
An easy calculation shows that
$\|f_z\|_{D_{\alpha}^p}\lesssim1$ and
$$f'_z(\omega)=\bar{z}\left(\frac{2(\alpha+2)-p}{p}\frac{(1-|z|^2)^{(\alpha+2)/p}}{(1-\bar{z}\omega)^{2(\alpha+2)/p}}-\frac{2(\alpha+2)}{p}\frac{(1-|z|^2)^{(\alpha+2)/p+1}}{(1-\bar{z}\omega)^{2(\alpha+2)/p+1}}\right),\, \omega\in\mathbb{D}.$$
Thus, we have $f_z(z)=0$ and $f'_z(z)=\frac{-\bar{z}}{(1-|z|^2)^{(\alpha+2)/p}}$\,.
Therefore,
$$\frac{|\bar{z}g(z)|}{(1-|z|^2)^{(\alpha+2)/p}}=|g'(z)f_z(z)+g(z)f'_z(z)|=|(g f_z)'(z)|\lesssim h(z)\,, \quad \text{ a.e. } [A_{\beta}]\,.$$
Hence, for $|z|>1/2$, it holds that
$$\frac{|g(z)|}{(1-|z|^2)^{(\alpha+2)/p}}\lesssim h(z)\,, \quad \text{ a.e. } [A_{\beta}]\,.$$

For $|z|\leq1/2$, it follows from the continuity of the function $\frac{1}{(1-|z|^2)^{(\alpha+2)/p}}$ that
$$\frac{1}{(1-|z|^2)^{(\alpha+2)/p}}\lesssim1\,.$$

Now, by taking the constant function $1$ and the monomial $z$ as the test function in $D_{\alpha}^p$, we get that $|g'(z)|\lesssim h(z)\text{ a.e. } [A_{\beta}]$, and $|g'(z)z+g(z)|\lesssim h(z)\text{ a.e. } [A_{\beta}]\,.$
Thus, for $|z|\leq1/2$, it also holds that
$$\frac{|g(z)|}{(1-|z|^2)^{(\alpha+2)/p}}\lesssim h(z)\,, \quad \text{ a.e. } [A_{\beta}]\,.$$

In conclusion, for all $z\in\mathbb{D}$,
$$\frac{|g(z)|}{(1-|z|^2)^{(\alpha+2)/p}}\lesssim h(z)\,, \quad \text{ a.e. } [A_{\beta}]\,,$$
which implies that
$$\int_{\mathbb{D}}\frac{|g(z)|^q}{(1-|z|^2)^{q(\alpha+2)/p}}dA_{\beta}<\infty\,.$$

The proof of (2) and (3) are similar to that of (1) by some minor modifications. For example, in (2), we take the test function
$$f_z(\omega)=\frac{\log(\frac{2}{1-\bar{z}\omega})}{\log(\frac{2}{1-|z|^2})^{1/p}}-\frac{\left(\log(\frac{2}{1-\bar{z}\omega})\right)^2}{\log(\frac{2}{1-|z|^2})^{1/p+1}}\,,\quad \omega\in\mathbb{D}\,.$$
Thus the proof is complete.
\end{proof}

By Theorem~\ref{th9}, Theorem~\ref{th10} and Theorem~\ref{th11}, we obtain the following corollary.
\begin{corollary}\label{cor4}
Let $-1<\alpha,\beta$, $g\in H(\mathbb D)$ and $0<p<\alpha+2,0<q<\infty$\,. Then the following statements are equivalent:\\
\phantom{(1)}(1) $T_g:D_{\alpha}^p\rightarrow D_{\beta}^q$ is order bounded;\\
\phantom{(1)}(2) $S_g:D_{\alpha}^p\rightarrow D_{\beta}^q$ is order bounded;\\
\phantom{(1)}(3) $M_g:D_{\alpha}^p\rightarrow D_{\beta}^q$ is order bounded;\\
\phantom{(1)}(4) $\int_{\mathbb{D}}\frac{|g(z)|^q}{(1-|z|^2)^{q(\alpha+2)/p}}dA_{\beta}<\infty$, that is, $g\in A^q_{\beta-\frac{q(\alpha+2)}{p}}$\,.
\end{corollary}


\begin{thebibliography}{9}
\bibitem{AC} A. Aleman, J. Cima,
{\it An integral operator on $H^p$ and Hardy's inequality}, J. Anal. Math. 85 (2001), 157-176.

\bibitem{AS} A. Aleman, A. Siskakis,
{\it An integral operator on $H^p$}, Complex Variables Theory Appl. 28 (1995), no.2, 149-158.

\bibitem{AS1} A. Aleman, A. Siskakis,
{\it Integration operators on Bergman spaces}, Indiana Univ. Math. J. 46 (1997), no.2, 337-356.

\bibitem{Car} L. Carleson,
{\it An interpolation problem for bounded analytic functions}, Amer. J. Math. 80 (1958), 921-930.

\bibitem{CP} O. Constantin, J. Pel\'{a}ez,
{\it Integral operators, embedding theorems and a Littlewood-Paley formula on weighted Fock spaces}, J. Geom. Anal.  26  (2016),  no. 2, 1109-1154.

\bibitem{ZCRZ} \v{Z}. \v{C}u\v{c}kovi\'{c}, R. Zhao,
{\it Weighted composition operators between different weighted Bergman spaces and different Hardy spaces}, Illinois J. Math. 51 (2007), 479-498.

\bibitem{Duren} P. Duren,
{\it Extension of a theorem of Carleson}, Bull. Amer. Math. Soc. 75 (1969), 143-146.

\bibitem{DS} P. Duren, A. Schuster,
{\it Bergman Spaces}, Math. Surveys Monogr., vol. 100, Amer. Math. Soc., Providence, RI (2004).

\bibitem{GKZ} Y. Gao, S. Kumar, Z. Zhou,
{\it Order bounded weighted composition operators mapping into the Dirichlet type spaces}, Chin. Ann. Math. Ser. B  37  (2016),  no. 4, 585-594.

\bibitem{GGP} P. Galanopoulos, D. Girela, J. Pel\'{a}ez,
{\it Multipliers and integration operators on Dirichlet spaces},  Trans. Amer. Math. Soc. 363 (2011),  no. 4, 1855-1886.

\bibitem{GP} D. Girela, J. Pel\'{a}ez,
{\it Carleson measures, multipliers and integration operators for spaces of Dirichlet type}, J. Funct. Anal. 241 (2006), no. 1, 334-358.

\bibitem{HKZ} H. Hedenmalm, B. Korenblum, K. Zhu,
{\it Theory of Bergman Spaces}, Grad. Texts in Math., vol. 199, Springer, New York (2000).

\bibitem{RH} R. Hibschweiler,
{\it Order bounded weighted composition operators}, Contemp. Math., 454, Amer. Math. Soc., Providence, RI, 2008.

\bibitem{HJ} H. Hunziker, H. Jarchow,
{\it Composition operators which improve integrability}, Math. Nachr.  152  (1991), 83-99.

\bibitem{Kumar} S. Kumar,
{\it Weighted composition operators between spaces of Dirichlet type}, Rev. Mat. Complut. 22 (2009), no. 2, 469-488.

\bibitem{LMN} J. Laitila, S. Miihkinen, P. Nieminen,
{\it Essential norms and weak compactness of integration operators}, Arch. Math. 97 (2011), no. 1, 39-48.

\bibitem{LLL} P. Li, J. Liu, Z. Lou,
{\it Integral operators on analytic Morrey spaces}, Sci. China Math. 57 (2014), no. 9, 1961-1974.

\bibitem{Li1} S. Li, S. Stevi\'{c},
{\it Generalized composition operators on Zygmund spaces and Bloch type spaces}, J. Math. Anal. Appl. 338 (2008), no. 2, 1282-1295.

\bibitem{Li2} S. Li, S. Stevi\'{c},
{\it Products of Volterra type operator and composition operator from $H^\infty$ and Bloch spaces to Zygmund spaces}, J. Math. Anal. Appl. 345 (2008), no. 1, 40-52.

\bibitem{LIN1} Q. Lin,
{\it Volterra Type Operators Between Bloch Type Spaces and Weighted Banach Spaces}, Integral Equations Operator Theory  91 (2019), no. 2, 91:13.

\bibitem{LIN} Q. Lin, J. Liu, Y. Wu,
{\it Volterra type operators on $S^p(\mathbb{D})$ spaces}, J. Math. Anal. Appl. 461 (2018), 1100-1114.

\bibitem{LIN2} Q. Lin, J. Liu, Y. Wu,
{\it Strict singularity of Volterra type operators on Hardy spaces}, J. Math. Anal. Appl. 492 (2020), no. 1, 124438, 9 pp.

\bibitem{Luecking} D. Luecking,
{\it Forward and reverse inequalities for functions in Bergman spaces and their derivatives}, Amer. J. Math. 107 (1985), 85-111.

\bibitem{TM} T. Mengestie,
{\it Product of Volterra type integral and composition operators on weighted Fock spaces}, J. Geom. Anal. 24 (2014), no. 2, 740-755.

\bibitem{TM1} T. Mengestie,
{\it Path connected components of the space of Volterra-type integral operators}, Arch. Math. 111 (2018), no. 4, 389-398.

\bibitem{PZ} J. Pau, R. Zhao,
{\it Carleson Measures, Riemann-Stieltjes and multiplication operators on a general family of function spaces}, Integr. Equ. Oper. Theory 78 (2014), 483-514.

\bibitem{P} Ch. Pommerenke,
{\it Schlichte Funktionen und analytische Funktionen von beschr\"{a}nkter mittlerer Oszillation}, (German) Comment. Math. Helv. 52 (1977), no. 4, 591-602.

\bibitem{AS} A.  Sharma,
{\it On order bounded weighted composition operators between Dirichlet spaces}, Positivity  21  (2017),  no. 3, 1213-1221.

\bibitem{SU} S.  Ueki,
{\it Order bounded weighted composition operators mapping into the Bergman space}, Complex Anal. Oper. Theory  6  (2012),  no. 3, 549-560.

\bibitem{WWG} S. Wang, M. Wang, X. Guo,
{\it Differences of Stevi\'{c}-Sharma operators}, Banach J. Math. Anal. 14 (2020), no. 3, 1019-1054.

\bibitem{wu} Z. Wu,
{\it Carleson measures and multipliers for Dirichlet spaces}, J. Funct. Anal. 169 (1999), 148-163.

\bibitem{XJ} J. Xiao,
{\it The $Q_p$ Carleson measure problem}, Adv. Math. 217 (2008), no. 5, 2075-2088.

\bibitem{zhao} R. Zhao,
{\it Pointwise multipliers from weighted Bergman spaces and Hardy spaces to weighted Bergman spaces}, Ann. Acad. Sci. Fenn. Math. 29 (2004), no. 1, 139-150.

\bibitem{zhu} K. Zhu,  {\it Operator Theory in Function Spaces}, Second edition. Mathematical Surveys and Monographs, 138. Amer. Math. Soc, Providence (2007).
\end{thebibliography}
\end{document}